\documentclass[11pt]{article}
\usepackage[T1]{fontenc}
\usepackage{lmodern}

\usepackage{amsmath,amssymb,amsthm,mathtools}
\usepackage{microtype}
\usepackage[a4paper,margin=26mm]{geometry}
\usepackage{enumitem,needspace}
\usepackage[hidelinks]{hyperref}
\usepackage[capitalise,nameinlink,noabbrev]{cleveref}
\hypersetup{pdftitle={Proper conflict-free degree-choosability below maximum average degree 8/3},pdfauthor={},pdfsubject={Graph coloring and maximum average degree}}
\newtheorem{theorem}{Theorem}[section]
\newtheorem{corollary}[theorem]{Corollary}
\newtheorem{conjecture}[theorem]{Conjecture}
\newtheorem{proposition}[theorem]{Proposition}
\newtheorem{lemma}[theorem]{Lemma}
\theoremstyle{remark}

\DeclareMathOperator{\mad}{mad}
\newcommand{\U}{\mathcal U}

\newcommand{\degtwo}{(\mathrm{degree}+2)}
\setlist[enumerate]{itemsep=2pt,topsep=5pt}
\title{Proper conflict-free degree-choosability of graphs
with maximum average degree less than $8/3$}

\author{
Huijuan Xu\thanks{Email: \texttt{xuhuijuan@zjnu.edu.cn}}\\[4pt]
School of Mathematical Sciences,
Zhejiang Normal University
}

\date{\today}
\begin{document}
\maketitle
\begin{abstract}
A proper coloring $\varphi$ of a graph $G$ is called proper conflict-free
if, for every non-isolated vertex $v$, some color appears exactly once
in its open neighborhood $N_G(v)$.
For a nonnegative integer $k$, a graph $G$ is proper conflict-free
$(\mathrm{degree}+k)$-choosable if, for every list assignment $L$
satisfying $|L(v)|\ge d_G(v)+k$ for each $v\in V(G)$, there exists
a proper conflict-free coloring $\varphi$ of $G$ with
$\varphi(v)\in L(v)$ for each $v\in V(G)$.
We prove that every connected graph other than the $5$-cycle $C_5$
with maximum average degree less than $8/3$ is proper conflict-free
$(\mathrm{degree}+2)$-choosable.
This improves the maximum average degree bound $18/7$ of
Kashima, \v{S}krekovski and Xu, and removes the
assumption $g(G)\ge7$ from the $8/3$ theorem of
Qi, Song and Cao.
We also show that the $(\mathrm{degree}+2)$ list-size bound is best possible
for graphs with maximum average degree arbitrarily close to $8/3$ from below.
As an immediate consequence, our theorem implies that every planar
graph of girth at least $8$ is proper conflict-free
$(\mathrm{degree}+2)$-choosable.
\end{abstract}


\section{Introduction}\label{sec:intro}

All graphs considered in this paper are finite, simple and undirected.
A coloring $\varphi$ of $G$ is \emph{proper} if adjacent vertices receive
distinct colors.
A proper coloring $\varphi$ of $G$ is \emph{proper conflict-free}
if, for every non-isolated vertex $v$, some color appears exactly
once in the neighborhood of $v$.
Fabrici, Lu\v{z}ar, Rindo\v{s}ov\'a and Sot\'ak~\cite{FLRS}
introduced this notion and investigated proper conflict-free colorings
of planar and outerplanar graphs.
Cho et al.~\cite{CCKP2025} further studied proper conflict-free
colorings of sparse graphs.

A \emph{list assignment} $L$ of a graph $G$ assigns to each vertex
$v$ a finite set $L(v)$ of permissible colors.
An \emph{$L$-coloring} of $G$ is a proper coloring $\varphi$
such that $\varphi(v)\in L(v)$ for every $v\in V(G)$.
A proper conflict-free $L$-coloring is an $L$-coloring that is
proper conflict-free.
For a nonnegative integer $k$, a graph $G$ is
\emph{proper conflict-free $(\mathrm{degree}+k)$-choosable}
if it admits a proper conflict-free $L$-coloring for every
list assignment $L$ satisfying
\[
|L(v)|\ge d_G(v)+k \qquad \text{for every }v\in V(G).
\]

A natural question is whether there exists a constant $k$
such that every graph is proper conflict-free
$(\mathrm{degree}+k)$-choosable.
For small values of $k$, it is known that $C_5$ is not proper
conflict-free $(\mathrm{degree}+2)$-choosable. Moreover, there are
infinitely many graphs that are not proper conflict-free
$(\mathrm{degree}+1)$-choosable~\cite[Proposition~4]{KSXouter}.

Motivated by these observations, Kashima, \v{S}krekovski and
Xu posed the following conjecture.
\begin{conjecture}[\cite{KSXouter,KSXlists}]
Every connected graph
other than $C_5$ is proper conflict-free $\degtwo$-choosable.
\end{conjecture}
They verified the conjecture for subcubic graphs~\cite{KSXlists}
and outerplanar graphs~\cite{KSXouter}.
Wang and Zhang~\cite{WZ} further verified the conjecture for
$K_4$-minor-free graphs of maximum degree at most $4$ and
outer-$1$-planar graphs of maximum degree at most $4$.

We consider this conjecture for graphs of bounded maximum
average degree.
The \emph{maximum average degree} of a nonempty graph $G$ is
\[
\mad(G)=
\max_{\substack{H\subseteq G\\ |V(H)|>0}}
\frac{2|E(H)|}{|V(H)|}.
\]
The \emph{girth} $g(G)$ is the length of a shortest cycle in $G$.
Kashima, \v{S}krekovski and Xu~\cite{KSXsparse} proved that
every connected graph $G$ other than $C_5$ with $\mad(G)<18/7$
is proper conflict-free $\degtwo$-choosable. More recently,
Qi, Song and Cao~\cite{QSC} proved that every graph $G$
with $\mad(G)<8/3$ and $g(G)\ge7$ is proper conflict-free
$\degtwo$-choosable.

Our main results are as follows.
The first theorem  removes the girth restriction from the
latter theorem while improving the maximum average degree
bound in the former theorem from $18/7$ to $8/3$.

\begin{theorem}\label{thm:main}
	Every connected graph $G$ other than $C_5$ with $\mad(G)<8/3$
	is proper conflict-free $\degtwo$-choosable.
\end{theorem}

The following result shows that the list-size bound in \cref{thm:main} is best possible.

\begin{theorem}\label{thm:sharp}
For every $\varepsilon>0$, there exists a connected graph
$G\not\cong C_5$ with $8/3-\varepsilon<\mad(G)<8/3$
that is not proper conflict-free
$(\mathrm{degree}+1)$-choosable.
\end{theorem}

For planar graphs, Wang and Zhang~\cite{WZ} proved that every planar
graph of girth at least $12$ is proper conflict-free
$\degtwo$-choosable.
Kashima, \v{S}krekovski and Xu~\cite{KSXsparse} improved the girth
bound to $9$, and Qi, Song and Cao~\cite{QSC} further improved it
to $8$. \cref{thm:main} also improves the girth bound to \(8\).

\begin{corollary}\label[corollary]{cor:planar}
	Every planar graph of girth at least $8$ is proper
	conflict-free $\degtwo$-choosable.
\end{corollary}

\begin{proof}
	Assume $G$ is a planar graph of girth at least $8$.
    We may assume that $G$ is connected (otherwise, we consider each connected component of $G$).
	Let $H$ be a nonempty connected subgraph of $G$.
	If $H$ is a tree, then
	\[
	\frac{2|E(H)|}{|V(H)|}<2.
	\]
	Otherwise, since $H$ is planar and has girth at least $8$, Euler's
formula implies that
	\[
	3|E(H)|\le4|V(H)|-8,
	\]
	and hence $2|E(H)|/|V(H)|<8/3$.
	Thus,  $\mad(G)<8/3$.
	Since $G\not\cong C_5$, by
	\cref{thm:main}, $G$ is  proper
	conflict-free $\degtwo$-choosable.
\end{proof}

To prove \cref{thm:main} without any girth restriction,
we overcome two main difficulties.

First, deleting a vertex $v$ together with the vertices of its
incident threads may cause some vertices to become isolated in the
remaining graph.
To overcome this, we choose $v$ to have minimum degree
among the vertices whose required charge transfer exceeds
their initial charge.
We show that any resulting isolated vertex leads either
to a reducible configuration or to a theta graph that
can be handled separately.

Second, two degree-two neighbors $v_1$ and $v_2$ of $v$
may have a common neighbor $u$ of degree three.
In this case, the colors of $v_1$ and $v_2$ cannot always be chosen
independently: they must not both receive the color of the third
neighbor of $u$, so that some color appears exactly once in
$N_G(u)$.
We handle this additional restriction by \cref{lem:coupled}.
Together, these arguments allow the coloring to be extended to the
deleted vertices without imposing any girth condition.

\section{Preliminaries}\label{sec:prelim}
 A $k$-vertex has degree $k$, and a $k^+$-vertex has degree at least $k$. 
A \emph{$j$-thread} of $G$ is a path $u_1\cdots u_j$ whose vertices all have degree two in $G$. Note that if
a graph contains a $k$-thread, then it also contains an $\ell$-thread for any $1 \le  \ell \le  k-1$.
A $k$-thread is called  maximal if it is not contained in a $(j+1)$-thread. 
A vertex $x$ is \emph{incident with} a maximal $k$-thread
$u_1\cdots u_k$ if $xu_1\cdots u_k$ is a path; in this case,
$x$ is called a \emph{boundary vertex} of the thread.

For a coloring $\varphi$ of a graph $H$, define
\[
 \U_\varphi(x,H)=
 \bigl\{c:\bigl|\{y\in N_H(x):\varphi(y)=c\}\bigr|=1\bigr\}.
\]
We say $x$ has a singleton color if $ \U_\varphi(x,H) \ne \emptyset$.

\begin{proposition}[{\cite{KSXouter}}]\label{prop:Cycle}
	Every connected outerplanar graph other than $C_5$ is proper conflict-free $\degtwo$-choosable.
\end{proposition}

\begin{proposition}[{\cite{KSXouter}}]\label{prop:C5}
	Let $C=v_0v_1v_2v_3v_4v_0$ be a $5$-cycle.
	Suppose that $L$ is a list assignment of $C$ with $|L(v_i)|\ge4$ for every $i\in\{0,\ldots,4\}$..
	Then $C$ is not proper conflict-free $L$-colorable if and only if
	$L(v_0)=L(v_1)=L(v_2)=L(v_3)=L(v_4)$
	and $|L(v_0)|=4$.
\end{proposition}

 Let $\Theta_{1,2^{\star a},3^{\star s},4^{\star t}}$ consist of two vertices joined by one edge and by $a+s+t$ internally disjoint paths.
 Here the length of a path means the number of its edges; among these paths, $a$ have length two, $s$ have length three, and $t$ have length four.

\begin{proposition}[{\cite{KSXsparse}}]\label[proposition]{prop:theta}
If $a,s,t$ are nonnegative integers with $a+s\ge1$ and $a+s+t\ge2$, then $\Theta_{1,2^{\star a},3^{\star s},4^{\star t}}$ is proper conflict-free $\degtwo$-choosable.
\end{proposition}

For the next two lemmas, let $X\subseteq V(G)$ and let $\varphi$
be a proper coloring of $G[X]$. Fix $v\in X$ and put
$A=N_G(v)\cap X$. Suppose that each $u\in N_G(v)-A$ has a set
$S(u)$ of at least two admissible colors such that $\varphi(w)\notin S(u)$
for every $w\in N_G(u)\cap X$.

The following two lemmas will be used to extend $\varphi$ to $N_G(v)-A$  so that $v$ has a singleton color.
\begin{lemma}\label[lemma]{lem:independent}
If $|A|\le3$ and $N_G(v)-A$ is a nonempty independent set,
then $\varphi$ can be extended to a proper coloring of
$G[X\cup N_G(v)]$ such that $\varphi(u)\in S(u)$ for every
$u\in N_G(v)-A$ and $v$ has a singleton color.
\end{lemma}

\begin{proof}

If some color $c$ appears exactly once in $A$, then for each
vertex $u\in N_G(v)-A$, let $\varphi(u)\in S(u)-\{c\}$.
Thus, $v$ has a singleton color $c$.

Now we  assume that no color appears exactly once in $A$.
Since$|A|\le3$, either  $A=\emptyset$ or all vertices
in $A$ are colored  with the same color.
Choose $x\in N_G(v)-A$ and let  $\varphi(x)\in S(x)$ such that   $\varphi(x)$
not used in $A$.
For each vertex $u\in N_G(v)-(A\cup\{x\})$, let
$\varphi(u)\in S(u)-\{\varphi(x)\}$.
Then $x$ is the only neighbor of $v$ colored with $\varphi(x)$, so $v$ has
a singleton color $\varphi(x)$.

Since  $N_G(v)-A$ is independent and $S(u)$ contains no color
used in  $N_G(u)\cap X$ for each $u\in N_G(v)-A$, the extended coloring is proper. 
\end{proof}

\begin{lemma}\label[lemma]{lem:coupled}
	Suppose that $|A|\le 2$,  $N_G(v)-A$ is an independent set and 
$|N_G(v)-A|\ge2$.

Let $x$ and $y$ be two distinct vertices in $N_G(v)-A$.
For any color $\delta$,  $\varphi$ can be extended to a proper coloring of
$G[X\cup N_G(v)]$ such that $\varphi(u)\in S(u)$ for every
$u\in N_G(v)-A$, $(\varphi(x),\varphi(y))\ne (\delta,\delta)$ and $v$ has a singleton color.
\end{lemma}

\begin{proof}
	Suppose first that $A=\emptyset$ or all vertices in $A$ are colored with  the same color.
Since $|S(x)|,|S(y)|\ge2$, let $\varphi(x)\in S(x)$ such that  $\varphi(x)$ is not used in $A$, and  let $\varphi(y)\in S(y)-\{\varphi(x)\}$.
Then let 
$\varphi(u)\in S(u)-\{\varphi(x)\}$ for each vertex
$u\in N_G(v)-(A\cup\{x,y\})$.
Thus,  $v$ has a singleton color $\varphi(x)$, and $x$ and $y$ receive
distinct colors.

Now assume that $|A|=2$ and the two vertices in $A$ are colored with  distinct
 colors, say $c$ and $e$.
If $\delta\in\{c,e\}$, then let $\varphi(u)\in S(u)-\{\delta\}$ for each vertex $u\in N_G(v)-A$.
Thus, $v$ has a singleton color $\delta$, and neither $x$ nor $y$ is colored with $\delta$.
	
Now suppose that $\delta\notin\{c,e\}$.
	If $S(x)-\{c\}\ne\{\delta\}$ or $S(y)-\{c\}\ne\{\delta\}$,
	then, by symmetry, we may assume
	that $S(x)-\{c\}\ne\{\delta\}$.
	Since $|S(x)|\ge2$, we have $S(x)-\{c,\delta\}\ne\emptyset.$
Let $\varphi(x)\in S(x)-\{c,\delta\}$, and let $\varphi(u)\in S(u)-\{c\}$ for  each vertex 
	$u\in N_G(v)-(A\cup\{x\})$.
	Thus, $v$ has a singleton color $c$, and $x$ is not colored with $\delta$.
	
If $S(x)-\{c\}=S(y)-\{c\}=\{\delta\}$, then   $S(x)=S(y)=\{c,\delta\}$.
Let $\varphi(x)=c$, $\varphi(y)=\delta$ and let $\varphi(u)\in S(u)-\{\delta\}$ for each vertex 
	$u\in N_G(v)-(A\cup\{x,y\})$.
Thus, 
	$y$ is the only neighbor of $v$ colored with $\delta$.
Hence, $v$ has a singleton color $\delta$, and $x$ and $y$ are not both colored with $\delta$.

Since  $N_G(v)-A$ is independent and $S(u)$ contains no color
used in  $N_G(u)\cap X$ for each $u\in N_G(v)-A$, the extended coloring is proper. 
\end{proof}

\section{Proof of \cref{thm:main}}\label{sec:main}

Suppose that \cref{thm:main} is false. Let $G$ be a counterexample such that $|V(G)|+|E(G)|$ is minimum.
Let $L$ be a list assignment of $G$ with
$|L(v)|\ge d_G(v)+2$ for every $v\in V(G)$, and suppose that
$G$ has no proper conflict-free $L$-coloring.
By \cref{prop:Cycle}, every cycle other than $C_5$ is proper
conflict-free $\degtwo$-choosable. Since $G\not\cong C_5$,
$G$ is not a cycle.

The following structural properties of $G$ follow from the local
reducibility arguments of Kashima, \v{S}krekovski and
Xu~\cite[Section~3.2]{KSXsparse}.
These arguments do not rely on
the maximum average degree of $G$. Therefore, these properties remain valid for our minimum
counterexample $G$ with $\mad(G)<8/3$.

Unless otherwise stated, all vertex degrees and threads are taken
with respect to $G$.
\begin{lemma}\label[lemma]{lem:structure}
The graph $G$ has the following properties.
\begin{enumerate}[label=\textup{(R\arabic*)},leftmargin=3.5em]
\item $\delta(G)\ge2$.

\item
Every maximal thread of $G$ is a $j$-thread for some
$j\in\{1,2,3\}$ and has two distinct $3^+$-vertices as boundary vertices.
If $j\in\{2,3\}$, then each boundary vertex of a maximal $j$-thread is a $4^+$-vertex.

\item Every $3^+$-vertex has a $3^+$-neighbor.
\item Every $k$-vertex, $k\ge4$, is incident with at most $k-3$ maximal $3$-threads.
\item A $k$-vertex, $k\ge4$, with exactly $k-1$ $2$-neighbors is incident with at most $k-4$ maximal $3$-threads.
\item If a $3$-vertex has two $2$-neighbors, then the other neighbor of each of these $2$-vertices has degree at least four.
\end{enumerate}
\end{lemma}

\begin{lemma}\label[lemma]{lem:diamond}
	There are no distinct vertices $v,z,p,q\in V(G)$ with $d_G(v)\ge4$ such that
	\[
	N_G(z)=\{v,p,q\},\qquad N_G(p)=N_G(q)=\{v,z\}.
	\]
\end{lemma}
\begin{proof}
	Let $H=G-\{z,p,q\}$.
	Since $G$ is connected and none of $z,p,q$ has a neighbor in  $V(G)-\{v,z,p,q\}$, $H$ is connected. 
	Then $H$ has a
	proper conflict-free $L$-coloring; if $H\cong C_5$, then it follows from
	 \Cref{prop:C5} and the fact that $H$ has a vertex $x$ with $|L(x)| \ge  d_G(x) + 2 \ge 
	d_H(x) + 3$, otherwise, it follows from the minimality of $G$.
	
Let  $\varphi$ be a proper conflict-free $L$-coloring of $H$.
	We extend $\varphi$ to an L-coloring  of $G$ by the following manner.	
	Since $d_H(v)=d_G(v)-3\ge1$, there exists a  color
	$\gamma_v\in\U_\varphi(v,H)$.
	Since $|L(z)|\ge5$ and 	$|L(p)|,|L(q)|\ge4$, let $\varphi(z)\in L(z)-\{\varphi(v),\gamma_v\}$, 
	$\varphi(p)\in L(p)-\{\varphi(v),\gamma_v,\varphi(z)\}$, and $\varphi(q)\in L(q)-\{\varphi(v),\gamma_v,\varphi(z)\}$.
    
    The resulting coloring is proper.
Since none of $z,p,q$ is colored with $\gamma_v$, $\gamma_v\in\U_\varphi(v,G)$.
Since $\varphi(v)\ne\varphi(z)$, we obtain that 
$\U_\varphi(p,G),\U_\varphi(q,G)\ne \emptyset$.
Since  none of $p,q$ is colored with $\varphi(v)$, $\varphi(v)\in\U_\varphi(z,G)$.
All other vertices retain their
singleton colors.
 Thus $\varphi$ is a proper conflict-free
$L$-coloring of $G$, a contradiction.
\end{proof}

For a $k$-vertex $x$ of $G$ with $k\ge4$, let $r(x)$ be the number of its $3^+$-neighbors in $G$. Let $a(x)$ and $b(x)$ count the incident maximal $1$-threads whose other boundary vertex is a $4^+$-vertex and a $3$-vertex, respectively. Let $s(x)$ and $t(x)$ count the incident maximal $2$-threads and maximal $3$-threads, respectively.
By \Cref{lem:structure} (R2), we have 
\[
 k=r(x)+a(x)+b(x)+s(x)+t(x).
\]
Define
\begin{align}
 W(x)&=a(x)+\tfrac32 b(x)+2s(x)+3t(x),\label{eq:W}\text{ and}\\
 E(x)&=W(x)-(3k-8)=8-3r(x)-2a(x)-\frac32 b(x)-s(x).\label{eq:excess}
\end{align}
We call $x$  \emph{overloaded} if $E(x)>0$.

\begin{lemma}\label[lemma]{lem:overloaded}
	 $G$ has an overloaded vertex.
\end{lemma}

\begin{proof}
Assume to the contrary that $G$ has no overloaded vertex.	
Thus,  $E(v)\le 0$ for each $4^+$-vertex $v$ of $G$.
	
	Let $\mu(v)=3d_G(v)-8$ be the initial charge of each vertex
	$v\in V(G)$.
	Since $\mad(G)<8/3$, we have
	\[
	\sum_{v\in V(G)}\mu(v)
	=6|E(G)|-8|V(G)|<0.
	\]
	
We redistribute the charge  according to the following discharging rules.

\begin{enumerate}[label=(\roman*)]
	\item  Let $u$ be a $3$-vertex of $G$.
	\begin{itemize}
		\item If $u$ has two $2$-neighbors, then $u$ sends charge $1/2$ to each of them.
		\item If  $u$ has one $2$-neighbor, then $u$ sends charge $1$ to  it.
	\end{itemize}

	\item Let $u$ be a $4^+$-vertex of $G$.
	\begin{itemize}
	\item $u$ sends charge $1$ to each vertex of an incident maximal $j$-thread for $j\in \{2,3\}$.
	
	\item 	$u$ sends charge $3/2$ to an adjacent $2$-vertex which has a $3$-neighbor.
	
	\item $u$ sends charge $1$ to an adjacent $2$-vertex which has a $4^+$-neighbor other than $u$.
	\end{itemize}
	
\end{enumerate}
	
	Let $\mu'(v)$ denote the final charge of  each vertex $v\in V(G)$. 
	
If $v$ is a $2$-vertex of $G$, then by  \Cref{lem:structure} (R2),  $v$ contained in  a maximal $j$-thread $p$ for some $j\in\{1,2,3\}$.
Let $x$ and $y$ be the two distinct boundary vertices of $p$.

If both $x$ and $y$ are $4^+$-vertices, then by (ii), each of $x$ and $y$ sends charge $1$ to $v$.
Hence, $\mu'(v)= \mu(v)+2=0$.
If both $x$ and $y$ are $3$-vertices, then by \Cref{lem:structure} (R2) and (R6), $j=1$ and 
both $x$ and $y$ have exactly one $2$-neighbor $v$.
Thus, by (i), each of $x$ and $y$ sends charge $1$ to $v$. Hence, $\mu'(v)= \mu(v)+2=0$.

Suppose that exactly one of $x$ and $y$ is a $4^+$-vertex and the other
is a $3$-vertex. By symmetry, we may assume that $x$ is a $4^+$-vertex
and $y$ is a $3$-vertex. By \Cref{lem:structure}(R2), we have $j=1$.
Then, by (i) and (ii), $x$ sends charge $3/2$ to $v$, while $y$ sends
charge at least $1/2$ to $v$. Hence,
$\mu'(v)\ge \mu(v)+2=0.$

If $v$ is a $3$-vertex of $G$, then by \Cref{lem:structure} (R3), $v$ has at most two $2$-neighbors.
Thus, by (i), we have $\mu'(v)\ge \mu(v)-1=0.$

If $v$ is a $4^+$-vertex of $G$, then  $v$ sends charge exactly $W(v)$ and receives
	no charge. Hence
	\[
	\mu'(v)=3d_G(v)-8-W(v)=-E(v)\ge0 .
	\]
Therefore, 
	\[
	0\le\sum_{v\in V(G)}\mu'(v)
	=\sum_{v\in V(G)}\mu(v)<0,
	\]
	a contradiction.
\end{proof}

By \Cref{lem:overloaded}, we choose an overloaded vertex $v$ of minimum degree among all overloaded
vertices, and put $k=d_G(v)\ge4$.
For brevity, write $r=r(v)$, $a=a(v)$, $b=b(v)$, $s=s(v)$ and $t=t(v)$.
Let $\mathcal P$ be the set of all maximal threads incident with $v$,
and let
$$
D_v=\{v\}\cup\bigcup_{p\in\mathcal P}V(p),\qquad H=G-D_v.
$$

\begin{lemma}\label[lemma]{lem:boundary}
$d_H(x)\ge1$ for each  vertex $x\in V(H)$.
\end{lemma}
\begin{proof}
Suppose that $d_H(x)=0$ for some $x\in V(H)$. Then
$N_G(x)\subseteq D_v$. Since $x\notin D_v$, \Cref{lem:structure} (R2)
implies that $d_G(x)\ge3$.
If $xv\notin E(G)$, all neighbors
of $x$ have degree two, contrary to \Cref{lem:structure} (R3).
Hence $xv\in E(G)$. 
If $d_G(x)=3$, then (R2) implies that
$x$ is incident with exactly two maximal $1$-threads in
$\mathcal P$, contrary to \cref{lem:diamond}. Thus $d_G(x)\ge4$.

Since
every maximal thread incident with $x$ also belongs to $\mathcal P$,
we obtain that
$$
r(x)=1,\quad b(x)=0,\quad a(x)\le a,\quad s(x)\le s,\quad t(x)\le t,
$$
and hence
$$
d_G(x)=1+a(x)+s(x)+t(x),\qquad E(x)=5-2a(x)-s(x).
$$
It follows that
$$
E(x)-E(v)=3(r-1)+2(a-a(x))+\tfrac32 b+(s-s(x))\ge0.
$$
Thus $x$ is overloaded. Since $d_G(x)\le1+(k-r)\le k$, the
choice of $v$ forces $d_G(x)=k$. Therefore,
$$
0=k-d_G(x)=(r-1)+b+(a-a(x))+(s-s(x))+(t-t(x)).
$$
All terms are nonnegative, so $r=1$, $b=0$, and every thread
in $\mathcal P$ has $x$ as its other boundary vertex.
Since $G$ is connected, 
$$
G\cong\Theta_{1,2^{\star a},3^{\star s},4^{\star t}}.
$$
By \Cref{lem:structure}(R5), $t\le k-4$, and hence
$a+s=(k-1)-t\ge3$. By \cref{prop:theta}, $G$ has a proper
conflict-free $L$-coloring, a contradiction.
\end{proof}


Each connected component $H'$ of $H$ has a proper conflict-free $L$-coloring; if $H'\cong C_5$
it follows from
	\Cref{prop:C5} and the fact that $H'$ has a vertex $x$ with $|L(x)| \ge  d_G(x) + 2 \ge 
	d_{H'}(x) + 3$, otherwise, it follows from the minimality of $G$.
Thus, $H$ has a proper conflict-free $L$-coloring $\varphi$.
We shall extend $\varphi$ to a proper conflict-free $L$-coloring  of $G$.


For each vertex $x\in V(H)$ with a neighbor in $D_v$, fix a color
$$
\gamma_x\in\U_\varphi(x,H),
$$
which exists by \cref{lem:boundary}.
The same color $\gamma_x$ is used whenever $x$ occurs as the boundary
vertex of more than one thread or is also adjacent to $v$.

Let $R$ be the set of $3^+$-neighbors of $v$, and let $X$ be the set
of vertices of $H$ incident with a maximal $1$-thread in $\mathcal P$.
The admissible colors for $v$ form the set            
\begin{equation}\label{eq:A}
A=L(v)-\left(
\bigcup_{x\in R}\{\varphi(x),\gamma_x\}
\;\cup\;\{\varphi(x):x\in X\}\right).
\end{equation}
Since $|R|=r$ and $|X|\le a+b$,
\begin{equation}\label{eq:Asize}
|A|\ge k+2-2r-a-b=s+t+2-r.
\end{equation}
 In particular, for any $\alpha\in A$, $\varphi$ can be extended to a proper $L$-coloring of $G[V(H)\cup\{v\}]$ such that $\varphi(v)=\alpha$ and $\gamma_x\in \U_\varphi(x,G[V(H)\cup\{v\}])$ for
every $x\in R$. Also, $\varphi(v)\ne\varphi(x)$ for every $x\in X$.

\subsection{Choices on threads}\label{subsec:threads}

For each vertex $u\in N_G(v)-R$, we define an admissible color
set $F_u(\alpha)\subseteq L(u)$ depending on the color $\alpha\in A$
assigned to $v$.
By \Cref{lem:structure}(R2), $u$ belongs to a unique maximal
$m$-thread in $\mathcal P$, where $m\in\{1,2,3\}$.
Throughout the following four lemmas, let $vu_1\cdots u_mx$ be a path
of $G$ with $p=u_1\cdots u_m\in\mathcal P$, and let
$$
H'=G[V(H)\cup\{v\}\cup V(p)].
$$
Thus $v,x$ are the boundary vertices of $p$.

\begin{lemma}\label[lemma]{lem-4+}
Suppose that $m=1$ and $d_G(x)\ge4$.
For every $\alpha\in A$, let
$$
F_{u_1}(\alpha)=L(u_1)-\{\alpha,\varphi(x),\gamma_x\}.
$$
Then the following hold.
\begin{enumerate}[label=(\roman*)]
\item $|F_{u_1}(\alpha)|\ge1$.
\item For every $c\in F_{u_1}(\alpha)$, $\varphi$ extends to a proper
$L$-coloring of $H'$ such that
$$
\varphi(v)=\alpha,\qquad \varphi(u_1)=c,\qquad
\U_\varphi(u_1,H')\ne\emptyset,\qquad
\gamma_x\in\U_\varphi(x,H').
$$
\end{enumerate}
\end{lemma}

\begin{proof}
 Since $|L(u_1)|\ge d_G(u_1)+2=4$,  $|F_{u_1}(\alpha)|\ge1$.
For $c\in F_{u_1}(\alpha)$, let $\varphi(v)=\alpha$ and
$\varphi(u_1)=c$. 
 The coloring is proper by \eqref{eq:A}
and the definition of $F_{u_1}(\alpha)$.

Since $\alpha\ne\varphi(x)$ by quation~\eqref{eq:A}, the two neighbors of $u_1$ receive
distinct colors and hence $\U_\varphi(u_1,H')\ne\emptyset$. 
Since  $c\ne\gamma_x$ and if $xv\in E(G)$, then
also $\alpha\ne\gamma_x$ by \eqref{eq:A}.
Hence, 
$\gamma_x\in\U_\varphi(x,H')$.
\end{proof}

 \begin{lemma}\label[lemma]{lem:threeboundary}
Suppose that $m=1$ and $d_G(x)=3$.
For every maximal $1$-thread $u\in\mathcal P$ incident with $x$ and
every $\alpha\in A$, let
$$
F_u(\alpha)=L(u)-\{\alpha,\varphi(x)\}.
$$
Then the following hold.
\begin{enumerate}[label=(\roman*)]
\item $|F_u(\alpha)|\ge2$ for every such $u$.
\item If $u_1$ is the only maximal $1$-thread in $\mathcal P$ incident
with $x$, then, for every $c\in F_{u_1}(\alpha)$, $\varphi$ extends
to a proper $L$-coloring of $H'$ with
$$
\varphi(v)=\alpha,\qquad \varphi(u_1)=c,\qquad
\U_\varphi(u_1,H')\ne\emptyset,\qquad
\U_\varphi(x,H')\ne\emptyset.
$$
\item If another maximal $1$-thread $u_1'\in\mathcal P$ is incident
with $x$, then $xv\notin E(G)$.
For any $c\in F_{u_1}(\alpha)$ and $c'\in F_{u_1'}(\alpha)$ with
$(c,c')\ne(\gamma_x,\gamma_x)$, $\varphi$ extends to a proper
$L$-coloring of $G[V(H')\cup\{u_1'\}]$ with
$$
\varphi(v)=\alpha,\qquad \varphi(u_1)=c,\qquad \varphi(u_1')=c',
$$
which is conflict-free at $u_1$, $u_1'$ and $x$.
\end{enumerate}
\end{lemma}
 
 \begin{proof}
 (i): Since $|L(u)|\ge d_G(u)+2=4$,  $|F_{u}(\alpha)|\ge2$.
 	
 (ii):	By \Cref{lem:structure} (R2), $u_1$ is the only
thread in $\mathcal P$ incident with $x$.
    For $c\in F_{u_1}(\alpha)$, let $\varphi(v)=\alpha$ and $\varphi(u_1)=c$.
The extending coloring is proper and since $\alpha\ne\varphi(x)$ by quation~\eqref{eq:A}, $\U_\varphi(u_1,H')\ne\emptyset$. 

If $xv\notin E(G)$, the two vertices of $N_H(x)$ have distinct
colors because $\varphi$ is conflict-free at $x$ in $H$.
If $xv\in E(G)$, then $N_H(x)=\{v_x\}$ for some vertex $v_x$,
with $\varphi(v_x)=\gamma_x\ne\alpha$ by \eqref{eq:A}.
In either case, at least two  neighbors
of $x$ receive distinct colors. Hence, $\U_\varphi(x,H')\ne\emptyset$. 
 	
(iii): 
By \cref{lem:boundary}, the third neighbor $v_x$ of $x$ contained in $V(H)$. Hence $xv\notin E(G)$,
$N_H(x)=\{v_x\}$ and $\varphi(v_x)=\gamma_x$.
For  $c\in F_{u_1}(\alpha)$ and $c'\in F_{u_1'}(\alpha)$ with
$(c,c')\ne(\gamma_x,\gamma_x)$, let $\varphi(v)=\alpha$, $\varphi(u_1)=c$ and $\varphi(u_1)=c'$.

Since $u_1$ and $u_1'$ are nonadjacent and
$c,c'\notin\{\alpha,\varphi(x)\}$, the coloring is proper. 
Since $\alpha\ne\varphi(x)$ by quation~\eqref{eq:A}, $u_1$ and $u_1$ is conflict-free.
Since $(c,c')\ne(\gamma_x,\gamma_x)$, at least two neighbors of $x$ receive distinct colors. Hence $x$ is conflict-free.
 \end{proof}
 
 The next two lemmas construct control sets for maximal $2$- and $3$-threads, respectively. All two-element sets in these
constructions are fixed before $\alpha$ is chosen.

\begin{lemma}\label[lemma]{lem:controltwo}
Suppose that $m=2$. Choose a two-element set
\[
T_p\subseteq L(u_2)-\{\varphi(x),\gamma_x\},
\] and 
call $T_p$ the \emph{control set} of $p$.

For every $\alpha\in A$, let
\begin{equation}\label{eq:Ftwo}
F_{u_1}(\alpha)=
\begin{cases}
L(u_1)-\{\alpha,\varphi(x)\},&\alpha\notin T_p,\\
L(u_1)-(T_p\cup\{\varphi(x)\}),&\alpha\in T_p.
\end{cases}
\end{equation}
Then the following hold.
\begin{enumerate}[label=(\roman*)]
\item  $|F_{u_1}(\alpha)|\ge1$. Moreover,
    if $\alpha\notin T_p$, then
    $|F_{u_1}(\alpha)|\ge2$.
\item For every $c\in F_{u_1}(\alpha)$, $\varphi$ extends to a proper
$L$-coloring of $H'$ such that
$$
\varphi(v)=\alpha,\qquad \varphi(u_1)=c,\qquad
\U_\varphi(u_i,H')\ne\emptyset\ (i=1,2),\qquad
\gamma_x\in\U_\varphi(x,H').
$$
\end{enumerate}
\end{lemma}

\begin{proof}
Since $|L(u_2)|\ge4$, the  set $T_p$ exists.
Since $|L(u_1)|\ge4$ and $|T_p|=2$,  (i) holds.

 For $c\in F_{u_1}(\alpha)$, let $\varphi(v)=\alpha$ and $\varphi(u_1)=c$.
If $\alpha\notin T_p$,  $T_p-\{\alpha,c\}\ne \emptyset$. If $\alpha\in T_p$, then $c\notin T_p$ and hence $T_p-\{\alpha,c\}\ne \emptyset$.
Thus, let 
$
\varphi(u_2)\in T_p-\{\alpha,c\}.$

It is clearly that the extending coloring is proper.
Since $\varphi(v)\ne \varphi(u_2)$ and $\varphi(u_1)\ne\varphi(x)$,
$\U_\varphi(u_i,H')\ne \emptyset$ for $i\in \{1,2\}$.
Since $\varphi(u_2)\ne\gamma_x$, and $\alpha\ne\gamma_x$ whenever
$xv\in E(G)$, we have $\gamma_x\in\U_\varphi(x,H')$.
\end{proof}

\begin{lemma}\label[lemma]{lem:controlthree}
Suppose that $m=3$.
Choose two-element sets
$$
B_p\subseteq L(u_3)-\{\varphi(x),\gamma_x\},\qquad
T_p\subseteq L(u_1)-B_p,
$$
and call $T_p$ the \emph{control set} of $p$.
For every $\alpha\in A$, define
\begin{equation}\label{eq:Fthree}
F_{u_1}(\alpha)=T_p-\{\alpha\}.
\end{equation}
Then the following hold.
\begin{enumerate}[label=(\roman*)]
\item $|F_{u_1}(\alpha)|\ge1$. Moreover, if
$\alpha\notin T_p$, then   $|F_{u_1}(\alpha)|=2$.
\item For every $c\in F_{u_1}(\alpha)$, $\varphi$ extends to a proper
$L$-coloring of $H'$ such that
$$
\varphi(v)=\alpha,\qquad \varphi(u_1)=c,\qquad
\U_\varphi(u_i,H')\ne\emptyset\ (i=1,2,3),\qquad
\gamma_x\in\U_\varphi(x,H').
$$
\end{enumerate}
\end{lemma}

\begin{proof}
Since $|L(u_1)|,|L(u_3)|\ge4$, $B_p$ and $T_p$ exist.
 Since $F_{u_1}(\alpha)=T_p-\{\alpha\}$ and $|T_P|=2$, (i) holds.

For $c\in F_{u_1}(\alpha)$, let $\varphi(v)=\alpha$ and $\varphi(u_1)=c$.
Since $|L(u_2)|\ge4$ and $|B_p|=2$, let 
$$\varphi(u_2)\in L(u_2)-\{\alpha,c,\varphi(x)\}, ~
\varphi(u_3)\in B_p-\{\varphi(u_2)\}.$$
It is clearly that the extending coloring of $H'$ is proper.
Since $T_p\cap B_p=\emptyset$, $\varphi(u_1)\ne \varphi(u_3)$ and hence $\U_\varphi(u_2,H')\ne\emptyset$.
Since 
$
\phi(v)\ne \phi(u_2)$ and $\phi(u_2)\ne\varphi(x)
$, $\U_\varphi(u_i,H')\ne\emptyset$ for $i\in \{1,3\}$.
Since  $\phi(u_3)\ne\gamma_x$, and $\alpha\ne\gamma_x$ if $xv\in E(G)$.
Thus $\gamma_x\in\U_\varphi(x,H')$.
\end{proof}

We now combine these local extensions.
Let
$$
K=G[V(H)\cup\{v\}\cup N_G(v)].
$$

By \Cref{lem:structure}(R2), each maximal thread in $\mathcal P$ has two distinct boundary vertices. Hence, distinct maximal threads in $\mathcal P$ have disjoint vertex sets, and each contains exactly one neighbor of $v$. Therefore,
$N_G(v)-R$
is an independent set.
For every $\alpha\in A$ and $u\in N_G(v)-R$, the set $F_u(\alpha)$
contains neither $\alpha$ nor any color used on $N_G(u)\cap V(H)$.
Thus, we extend $\varphi$ to a proper $L$-coloring of $K$ by coloring $v$ with an arbitrary color $\alpha\in A$ and each $u\in N_G(v)-R$ with an arbitrary color from $F_u(\alpha)$.

\begin{lemma}\label[lemma]{lem:extension}
Suppose that $\varphi$ has been extended to a proper $L$-coloring
of $K$ with $\varphi(v)=\alpha\in A$ and
$\varphi(u)\in F_u(\alpha)$ for every $u\in N_G(v)-R$.
Assume also that
\begin{equation}\label{eq:coupledcondition}
(\varphi(u),\varphi(u'))\ne(\gamma_x,\gamma_x)
\end{equation}
whenever a $3$-vertex $x\in V(H)$ is incident with two maximal
$1$-threads $u,u'\in\mathcal P$.
Then $\varphi$ extends to a proper $L$-coloring of $G$ in which
every vertex other than $v$ has a singleton color.
\end{lemma}

\begin{proof}
We apply
\cref{lem-4+,lem:threeboundary,lem:controltwo,lem:controlthree}
to the threads in $\mathcal P$ to complete the coloring of $G$.
Since there are no edges between distinct threads, the resulting
coloring is proper. These lemmas also ensure that every vertex
of $D_v-\{v\}$ has a singleton color.

It is clearly that 
all vertices of $H$ with no neighbor in $D_v$ retain their
singleton colors. 
Now consider a vertex $x\in V(H)$ with a neighbor in $D_v$.

If $x$ is not incident with any thread in $\mathcal{P}$, then $xv\in E(G)$.
By \eqref{eq:A}, $v$ avoids $\gamma_x$ and hence $\gamma_x$ remains a singleton color at $x$.

If $d_G(x)=3$ and $x$ is incident with a thread in $\mathcal P$,
then every such thread is a maximal $1$-thread by
\Cref{lem:structure} (R2).
Thus, $x$ has a singleton color by \cref{lem:threeboundary}
and \eqref{eq:coupledcondition}.

If  $d_G(x)\ge4$ and  $x$ is incident with a thread in $\mathcal P$, then every new neighbor of $x$ avoids $\gamma_x$ by
\cref{lem-4+,lem:controltwo,lem:controlthree} and \eqref{eq:A}.
Hence $\gamma_x$ remains a singleton color at $x$.
\end{proof}

As shown above, we obtain a proper $L$-coloring of $K$ by coloring $v$ with any $\alpha\in A$ and each $u\in N_G(v)-R$ with any color from $F_u(\alpha)$. If $v$ has a singleton color in $K$ and
\eqref{eq:coupledcondition} holds, then, by \cref{lem:extension},
this coloring extends to a proper conflict-free $L$-coloring of
$G$, since all neighbors of $v$ already lie in $K$.
It therefore remains to choose these colors so that both
conditions are satisfied.

\subsection{Choosing a singleton color at $v$}
\label{subsec:singleton}

We determine which neighbors of $v$ may have only one
admissible color once $\alpha$ is fixed.
By the preceding lemmas, every neighbor $u$ of $v$ has at least two
admissible colors, except in the following cases:
\begin{itemize}
	\item $u$ is a direct $3^+$-neighbor of $v$, whose color is
	already fixed by $\varphi$;
	
	\item $u$ is the vertex of a maximal $1$-thread whose other
	boundary vertex is a $4^+$-vertex; or
	
	\item $u$ is contained in a maximal $2$- or $3$-thread $p$ and
	$\alpha\in T_p$.
\end{itemize}

By \Cref{lem:structure} (R3), we have $r\ge1$.
Since $E(v)>0$, by \eqref{eq:excess},  we obtain that $r\in\{1,2\}$.

\medskip
\noindent\textbf{Case 1.} $r=1$.

By \Cref{lem:structure} (R5), $t\le k-4$, and hence
$$
 a+b+s\ge3,\qquad 2a+\tfrac32 b+s<5.
$$
It follows that $a\le1$ and $b\le3$; moreover, $b\ge2$ implies $a=0$.

By \eqref{eq:Asize}, $|A|\ge s+t+1$.  Since there are $s+t$ two-element
control sets, the pigeonhole principle implies that some
$\alpha\in A$ belongs to at most one control set.
We color $v$ with $\alpha$ and color each vertex of $N_G(v)-R$ whose
admissible color set has size one.

Suppose first that no $3$-vertex is incident with two maximal
$1$-threads in $\mathcal{P}$.
Since $r=1$, $a\le1$, and $\alpha$ belongs to at most one control set,
at most three neighbors of $v$ are colored.
Every other neighbor of $v$ has at least two admissible colors
with no additional restriction.
Since $k\ge4$, at least one  neighbor of $v$ remains uncolored.
By \Cref{lem:independent}, $\varphi$ extends to a proper
$L$-coloring of $K$ in which $v$ has a singleton color.

Suppose now that a $3$-vertex $x$ is incident with two maximal
$1$-threads $u,u'\in\mathcal P$.
Then $b\ge2$ and hence $a=0$. 
Since $b\le3$, $x$ is the unique such
$3$-vertex. 
Thus, at most two neighbors of $v$ are now colored, and every
remaining neighbor has at least two admissible colors.
In particular, $u,u'$ remain uncolored.
By \cref{lem:coupled}, with $\delta=\gamma_x$, $\varphi$ extends
to a proper $L$-coloring of $K$ in which $v$ has a singleton color
and \eqref{eq:coupledcondition} holds.

\medskip
\noindent\textbf{Case 2.} $r=2$.

By \Cref{lem:structure} (R4), $t\le k-3$, and hence
$a+b+s=k-2-t\ge1$.
Since $E(v)>0$, we have
$2a+\frac32 b+s<2.$
Therefore,
\begin{equation}\label{eq:patterns}
	(a,b,s)=(0,1,0)
	\quad\text{or}\quad
	(a,b,s)=(0,0,1).
\end{equation}
In particular, no $3$-vertex is incident with two maximal
$1$-threads in $\mathcal P$. By \eqref{eq:Asize},
 $|A|\ge s+t$.
 
 Suppose first that some $\alpha\in A$ belongs to at most one
control set. Color $v$ with $\alpha$, and color each vertex of
$N_G(v)-R$ whose admissible set has size one.
Since $a=0$ and $r=2$, at most three neighbors of $v$ are now
colored and  every remaining neighbor has at least two admissible
colors. Since $k\ge4$,  at least one neighbor of $v$ remains uncolored.
By \Cref{lem:independent}, $\varphi$ extends to a proper
$L$-coloring of $K$ in which $v$ has a singleton color.

Now assume that every $\alpha\in A$ belongs to at least two
control sets. Since $|A|\ge s+t$, we have
$$
2(s+t)\le2|A|\le\sum_p|T_p\cap A|
\le\sum_p|T_p|=2(s+t),
$$
where both sums run over all maximal $2$- and $3$-threads
$p\in\mathcal P$.
Hence $|A|=s+t$, every $T_p$ is contained in $A$, and every color
of $A$ belongs to exactly two control sets.

Let $z_1$ and $z_2$ be the two $3^+$-neighbors of $v$.
Since $|A|=s+t$, equality holds in \eqref{eq:Asize}.
It follows from the definition of $A$ that
\[
\varphi(z_1),\varphi(z_2)\notin A
\qquad\text{and}\qquad
\varphi(z_1)\ne\varphi(z_2).
\]

For each maximal thread $p\in\mathcal P$, let $u_p$ be the vertex of
$p$ adjacent to $v$.

\medskip
\noindent\textbf{Subcase 2.1.} $(a,b,s)=(0,1,0)$.

Choose any $\alpha\in A$ and color $v$ with $\alpha$.
Let $u$ be the vertex of the unique maximal $1$-thread in
$\mathcal P$. Since $|F_u(\alpha)|\ge2$, assign to $u$ a color
in $F_u(\alpha)-\{\varphi(z_1)\}$.
For each maximal $3$-thread $p\in\mathcal P$, assign to $u_p$
any color in $F_{u_p}(\alpha)$.
By \eqref{eq:Fthree},
$$
F_{u_p}(\alpha)=T_p-\{\alpha\}\subseteq A,
$$
so none of these vertices is colored with $\varphi(z_1)$.
Thus, $\varphi$ extends to a proper
$L$-coloring of $K$ in which $v$ has a singleton color $\varphi(z_1)$.

\medskip
\noindent\textbf{Subcase 2.2.} $(a,b,s)=(0,0,1)$.

Let $p_0\in \mathcal{P}$ be the unique maximal $2$-thread and let $T_{p_0}$ be
its control set.
Since $k\ge4$, $t=k-3\ge1$.

If $t\ge2$, then $|A|=s+t\ge3$. Color $v$ with a color
$\alpha\in A-T_{p_0}$.
By \cref{lem:controltwo},
$|F_{u_{p_0}}(\alpha)|\ge2$, so we can assign to $u_{p_0}$ a color
in $F_{u_{p_0}}(\alpha)-\{\varphi(z_1)\}$.
For each maximal $3$-thread $p\in\mathcal P$, assign to $u_p$
any color in $F_{u_p}(\alpha)$. By \eqref{eq:Fthree},
$F_{u_p}(\alpha)\subseteq T_p\subseteq A$, so none of these
vertices is colored with $\varphi(z_1)$.
Therefore, $v$ has a singleton color $\varphi(z_1)$.

Now assume that $t=1$. Then $|A|=s+t=2$; write
$A=\{\alpha,\beta\}$. 
Hence, $T=A$ for each control set $T$.
We color $v$ with $\alpha$.
Let $p_3$ be the unique maximal
$3$-thread. By \eqref{eq:Fthree},
$F_{u_{p_3}}(\alpha)=\{\beta\}$, so we color $u_{p_3}$ with $\beta$.
Since $\alpha\in T_{p_0}=A$, by \eqref{eq:Ftwo}, we obtain that
$$
F_{u_{p_0}}(\alpha)\cap A=\emptyset.
$$
We assign to $u_{p_0}$ any color in $F_{u_{p_0}}(\alpha)$.
Since $\varphi(z_1),\varphi(z_2)\notin A$, $v$ has a singleton color  $\beta$.

In all cases, $v$ has a singleton color and
\eqref{eq:coupledcondition} is satisfied. Hence, by
\cref{lem:extension}, $\varphi$ extends to a proper conflict-free
$L$-coloring of $G$, a contradiction.

\section{Proof of \cref{thm:sharp}}
\label{sec:sharp}
Fix $\varepsilon>0$. Let $t$ be a positive integer such that
$8/(9t+3)<\varepsilon$.
Let $C_i=x_i a_i b_i c_i x_i$, $1\le i\le t$, be pairwise
vertex-disjoint $4$-cycles. Let $B_t$ be obtained from their
disjoint union by identifying $x_1,\ldots,x_t$ into a single
vertex $x$.
Then $B_t\ncong C_5$ is a connected  graph and the average degree of $B_t$ is $8t/(3t+1)$.

Let $H$ be a nonempty connected subgraph of $B_t$.
If $H$ is acyclic, then its average degree is less than $2$.
Now  we assume that $H$ contains at least one cycle.
By the construction of $B_t$, we obtain that every cycle of $H$ is one of $C_1,\ldots,C_t$.
Assume $H$ contains 
$r$ cycles, where $1\le r\le t$.
Thus $|V(H)|\ge3r+1$.
Since deleting one edge from each of the $r$ cycles leaves a tree
on $V(H)$, we have $|E(H)|=|V(H)|-1+r$. Hence
\[
\frac{2|E(H)|}{|V(H)|}
=2+\frac{2r-2}{|V(H)|}
\le 2+\frac{2r-2}{3r+1}
=\frac{8r}{3r+1}
\le\frac{8t}{3t+1}.
\]
Since the same bound holds for disconnected subgraphs of $B_t$ by considering
their components,
 we obtain that $\mad(B_t)=8t/(3t+1)$.
Therefore, 
\[
0<\frac83-\mad(B_t)=\frac{8}{9t+3}<\varepsilon.
\]

Next, we will show that $B_t$ is not proper conflict-free
$(\mathrm{degree}+1)$-choosable.
For $1\le i\le t$, let $S_i=\{1,2i,2i+1\}$.
Define a list assignment $L$ of $B_t$ by
\[
L(x)=\bigcup_{i=1}^{t}S_i=\{1,\ldots,2t+1\},
\qquad
L(a_i)=L(b_i)=L(c_i)=S_i.
\]
Then we have 
 $|L(v)|=d_{B_t}(v)+1$ for every vertex $v\in V(B_t)$.

Suppose that $\varphi$ is a proper conflict-free $L$-coloring
of $B_t$.
Fix $i\in \{1,2,\dots,t\}$. Since $N_{B_t}(b_i)=\{a_i,c_i\}$, 
$\varphi(a_i)\ne\varphi(c_i)$.
It  implies that 
$S_i=\{\varphi(a_i),\varphi(b_i),\varphi(c_i)\}$.
Since $N_{B_t}(a_i)=\{b_i,x\}$, $\varphi(x)\ne\varphi(b_i)$.
Since 
$\varphi(x)\notin\{\varphi(a_i),\varphi(c_i)\}$,
$\varphi(x)\notin S_i$.
Thus,  $\varphi(x)\notin L(x)$, a contardiction.
Therefore, $B_t$ has no proper conflict-free $L$-coloring and hence $B_t$ is not proper conflict-free
$(\mathrm{degree}+1)$-choosable.\qed

\section*{Statements and Declarations}

\noindent\textbf{Funding.}
The authors received no funding for this work.

\medskip
\noindent\textbf{Competing interests.}
The authors declare no competing interests.

\medskip
\noindent\textbf{Data Availability.}
No data were used in this study.

\bibliographystyle{plain}
\bibliography{reference}

@article{FLRS,
  title={Proper conflict-free and unique-maximum colorings of planar graphs with respect to neighborhoods},
  author={Fabrici, Igor and Lu{\v{z}}ar, Borut and Rindo{\v{s}}ov{\'a}, Simona and Sot{\'a}k, Roman},
  journal={Discrete Applied Mathematics},
  volume={324},
  pages={80--92},
  year={2023},
  publisher={Elsevier}
}

@article{CCKP2025,
  author  = {Cho, Eun-Kyung and Choi, Ilkyoo and
             Kwon, Hyemin and Park, Boram},
  title   = {Proper conflict-free coloring of sparse graphs},
  journal = {Discrete Applied Mathematics},
  volume  = {362},
  pages   = {34--42},
  year    = {2025},
  doi     = {10.1016/j.dam.2024.11.016}
}

@article{KSXlists,
	title={Results on proper conflict-free list coloring of graphs},
	author={Kashima, Masaki and {\v{S}}krekovski, Riste and Xu, Rongxing},
	journal={Discrete Mathematics},
	volume={349},
	number={12},
	pages={115330},
	year={2026},
	publisher={Elsevier}
}

@article{KSXouter,
	title={Proper conflict-free degree-choosability of outerplanar graphs},
	author={Kashima, Masaki and {\v{S}}krekovski, Riste and Xu, Rongxing},
	journal={Discrete Mathematics},
	volume={349},
	number={3},
	pages={114800},
	year={2026},
	publisher={Elsevier}
}

@article{KSXsparse,
	title={Degree-choosability of proper conflict-free list coloring of sparse graphs},
	author={Kashima, Masaki and {\v{S}}krekovski, Riste and Xu, Rongxing},
	journal={arXiv preprint arXiv:2601.15611},
	year={2026}
}

@article{QSC,
	title={Proper conflict-free choosability of sparse graphs with girth at least seven},
	author={Qi, Xingqin and Song, Huimin and Cao, Zhulou},
	journal={arXiv preprint arXiv:2609.11249},
	year={2026}
}

@article{WZ,
	title={Proper conflict-free choosability of planar graphs},
	author={Wang, Yuting and Zhang, Xin},
	journal={arXiv preprint arXiv:2512.22805},
	year={2025}
}
\end{document}